\documentclass[a4paper,12p]{article}
\title{Statistical immersions between statistical manifolds of constant curvature}
\author{\textrm{\textsuperscript{a}Chol Rim Min, Song Ok Choe, Yun Ho An} $~~$ $~~$\\  \\
                  {\small{Faculty of Mathematics, \textbf{Kim Il Sung} University, Pyongyang, D. P. R. Korea}}\\
		  {\small{\textsuperscript{a}Corresponding author. e-mail address: mincholrim@yahoo.com}}
}
    
\date{}  
\usepackage{latexsym}   
\usepackage[pdftex]{graphicx}
\usepackage{amsmath, amsthm, amsfonts, amssymb}
\usepackage{hyperref}

\numberwithin{equation}{section}
\newtheorem{Thot}{Theorem}[section]

\newtheorem{Coro}{Corollary}[section]
\newtheorem{Def}{Definition}[section]
\newtheorem{lemm}{Lemma}[section]

\begin{document}
\maketitle      
\begin{abstract}   
The condition for the curvature of a statistical manifold to admit a kind of standard hypersurface is given. We study the statistical hypersurfces of some types of the statistical manifolds $(M, \nabla, g )$, which enable $(M, \nabla^{(\alpha)}, g ),  \forall\alpha\in\mathbf{R}$ to admit the structure of a constant curvature.\\
\end{abstract}

{\bf Keywords:} statistical manifold of a constant curvature, statistical submanifold, Hessian structure, statistical hypersurface

{\bf MSC 2010:}53A15, 53C42, 53B05

\section{Introduction}
\quad 

Since Lauritzen introduced the notation of statistical manifolds in 1987 \cite{laur}, the geometry of statistical manifolds has been developed in close relations with affine differential geometry and Hessian geometry as well as information geometry (see, for example, \cite{amari,kuro,zhang}). In this paper we study the hypersurface of statistical manifolds.

 Let $M$ be an $n$-dimensional manifold, $\nabla$  a torsion-free affine connection on $M$, $g$ a Riemannian metric on  $M$, and $R$  a curvature tensor field on  $M$. We denote by $TM$  the set of vector fields on $M$, and by  $TM^{(r, s)}$ the set of tensor fields of type $(r, s)$  on $M$. 
\begin{Def}\label{def1_1}
\textnormal{A pair} $(\nabla, g)$ \textnormal{is called a} statistical structure \textnormal{on} $M$ \textnormal{if} $(M, \nabla, g )$ \textnormal{is a statistical manifold, that is,}$\nabla$ \textnormal{is a torsion-free affine connection and for all }$X, Y, Z\in T(M), ~ (\nabla_{X}g)(Y,Z)=(\nabla_{Y}g)(X,Z).$
\end{Def}

Let $\nabla^{\circ}$  be a Levi-Civita connection of  $g$. Certainly, a pair $(\nabla^{\circ}, g)$  is a statistical structure, which is called a Riemannian statistical structure or a trivial statistical structure (see \cite{furu}).

On the other hand, the statistical structure is also introduced from affine differential geometry which was proposed by Blasche (see \cite{nomi}). 
Recently the relation between statistical structures and Hessian geometry has been studied (see \cite{furu,shima}).

For all $\alpha\in\mathbf{R}$, a connection $\nabla^{(\alpha)}$ is defined by
\begin{equation*}
	\nabla^{(\alpha)}=\frac{1+\alpha}{2}\nabla+\frac{1-\alpha}{2}\nabla^{*}
\end{equation*}
where $\nabla$  and $\nabla^{*}$  are dual connections on $M$. 
We study a statistical hypersurface of a statistical manifold $(M, \nabla, g )$ which enables $(M, \nabla^{(\alpha)}, g ),  \forall\alpha\in\mathbf{R} $  to admit the structure of a constant curvature.

In section 3, a statistical manifold $(M, \nabla, g )$, which enables $(M, \nabla^{(\alpha)}, g ),  \forall\alpha\in\mathbf{R}$ to admit the structure of a constant curvature, is considered. 
In section 4, we study characteristics of statistical immersions between statistical manifolds $(M, \nabla, g )$ which enable  $(M, \nabla^{(\alpha)}, g ),  \forall\alpha\in\mathbf{R}$  to admit the structure of a constant curvature. 

                                              
\section{Preliminaries}
\quad 

A statistical manifold $(M, \nabla, g)$  is said to be of constant curvature  $k\in\mathbf{R}$ if 
\begin{equation}
 R(X, Y)Z=k\{ g(Y, Z)X-g(X, Z)Y\} , \forall X, Y,Z\in TM \tag{2.1}
 \end{equation}
holds, where  $R$ is the curvature tensor field of $\nabla$. A pair $(\nabla, g)$  is called a Hessian structure if a statistical manifold  $(M, \nabla, g)$ is of constant curvature 0.

A Riemannian metric $g$  on a flat manifold $(M,  g)$  is called a Hessian metric if $g$ can be locally expressed by 
\begin{equation*}
 g=Dd\varphi, 
 \end{equation*}
that is, 
\begin{equation*}
 g_{ij}=\frac{\partial ^{2}\varphi}{\partial x^{i}\partial x^{j}}, 
 \end{equation*}
 where  $\{ x^{1} , \cdot \cdot \cdot , x^{n} \}$ is an affine coordinate system with respect to $\nabla$. Then $(M, \nabla, g)$  is called a Hessian manifold (see \cite{shima}).

Let $(M, \nabla, g)$  be a Hessian manifold and $K(X, Y) :=\nabla _{X}Y-\nabla ^{\circ}_{X}Y$  be the difference tensor between the Levi-Civita connection $\nabla ^{\circ}$  of  $g$ and $\nabla$. A covariant differential of differential tensor $K$ is called a Hessian curvature tensor for $(\nabla, g)$. A Hessian manifold  $(M, \nabla, g)$ is said to be of constant Hessian curvature $c\in\mathbf{R}$   if 
\begin{equation*}
(\nabla _{X}K)(Y, Z)=-\frac{c}{2}\{ g(X, Y)Z+g(X, Z)Y\} ,\forall X,Y,Z\in TM
 \end{equation*}
holds (see \cite{shima}).\\ 


\noindent\textbf{Example 2.1.}$(\cite{furu})$ 
 Let $(H, \tilde g)$  
 be the upper half space:
\begin{equation*}
H:=\left\{ \left. y=(y^{1}, \cdots, y^{n+1})^{T}\in \mathbf{R}^{n+1} \right\vert y^{n+1}>0 \right\} , \tilde g :=(y^{n+1})^{-2} \sum_{i=1}^{n+1}dy^{i}dy^{i}.
\end{equation*}

We define an affine connection $\tilde \nabla$   on $H$  by the following relations:
\begin{equation*}
\tilde \nabla_{\frac{\partial }{\partial y^{n+1}}}\frac{\partial }{\partial y^{n+1}}=(y^{n+1})^{-1}\frac{\partial }{\partial y^{n+1}}, \tilde \nabla_{\frac{\partial }{\partial y^{i}}}\frac{\partial }{\partial y^{j}}=2\delta _{ij}(y^{n+1})^{-1}\frac{\partial }{\partial y^{n+1}},
\end{equation*}
\begin{equation*}
\tilde \nabla_{\frac{\partial }{\partial y^{i}}}\frac{\partial }{\partial y^{n+1}}=\tilde \nabla_{\frac{\partial }{\partial y^{n+1}}}\frac{\partial }{\partial y^{j}}=0,
\end{equation*}
where $i, j=1, \cdot \cdot \cdot, n$. Then $(H, \tilde \nabla, \tilde g)$  is a Hessian manifold of constant Hessian curvature 4.

Let  $(\tilde M, \tilde \nabla, \tilde g)$ be a statistical manifold and  $f:M\rightarrow  \tilde M$ be an immersion. We define $g$  and $\nabla$  on  $M$ by 
\begin{equation*}
g=f^{*}\tilde g , \quad g(\nabla_{X}Y, Z)=\tilde g(\tilde \nabla _{X}f_{*}Y,f_{*}Z), \quad \forall X, Y, Z\in TM.
\end{equation*}
Then the pair $(\nabla, g)$  is a statistical structure on $M$, which is called the one by $f$  from $(\tilde \nabla, \tilde g)$ (see \cite{furu}).  
 
Let $(M, \nabla, g)$  and  $(\tilde M, \tilde \nabla, \tilde g)$ be two statistical manifolds. An immersion $f:M\rightarrow  \tilde M$  is called a statistical immersion if $(\nabla, g)$  coincides with the induced statistical structure (see \cite{furu}).

Let $f:(M, \nabla, g)\rightarrow  (\tilde M, \tilde \nabla, \tilde g)$   be a statistical immersion of codimension one, and  $\xi$  a unit normal vector field of $f$. Then we define  $h, h^{*}\in TM^{(0, 2)}$, $\tau,\tau^{*}\in TM^{*}$  and  $A, A^{*}\in TM^{(1, 1)}$ by the following Gauss and Weingarten formulae:
\begin{equation*}
\tilde\nabla_{X}f_{*}Y=f_{*}\nabla_{X}Y+h(X, Y)\xi , \quad \tilde\nabla_{X}\xi=-f_{*}A^{*}X+\tau^{*}(X)\xi,
\end{equation*}
\begin{equation*}
\tilde\nabla^{*}_{X}f_{*}Y=f_{*}\nabla^{*}_{X}Y+h^{*}(X, Y)\xi  ,\quad \tilde\nabla^{*}_{X}\xi=-f_{*}AX+\tau(X)\xi,\quad \forall X,Y\in TM,
\end{equation*}
where $\tilde\nabla^{*}$  is the dual connection of $\tilde\nabla$  with respect to  $\tilde g$.

In addition, we define $II \in TM^{(0, 2)}$  and  $S \in TM^{(1, 1)}$  by using the Riemannian Gauss and Weingarten formulae:
\begin{equation*}
\tilde\nabla^{*}_{X}f_{*}Y=f_{*}\nabla^{*}_{X}Y+II(X, Y)\xi ,\quad \tilde\nabla^{*}_{X}\xi=-f_{*}SX.
\end{equation*}

For more details on the Gauss, Codazzi and Ricci formulae on statistical hypersurfaces, we refer to \cite{furu}.

                                           
\section{The condition that a statistical manifold  $(M, \nabla^{(\alpha)}, g )$ is of constant curvature for any $\alpha\in\mathbf{R}$ }
\quad

In this section we consider a condition that a statistical manifold  $(M, \nabla^{(\alpha)}, g )$ is of constant curvature for any $\alpha\in\mathbf{R}$ .
\begin{Thot}\label{theo3_1}
A statistical manifold $(M, \nabla^{(\alpha)}, g )$ is of constant curvature for any $\alpha\in\mathbf{R}$  iff there exist  $\alpha_{1}, \alpha_{2}\in\mathbf{R}(|\alpha_{1}|\neq|\alpha_{2}|) $  such that statistical manifolds $(M, \nabla^{(\alpha_{1})}, g )$  and $(M, \nabla^{(\alpha_{2})}, g )$  are of constant curvature. 
\end{Thot}

\begin{proof}
Necessity is obvious. We find sufficiency.
Without loss of generality, we assume $\alpha_{1}\neq 0$. Then since 
\begin{equation*}
	\nabla^{(\alpha)}=\frac{\alpha_{1}+\alpha}{2\alpha_{1}}\nabla^{(\alpha_{1})}+\frac{\alpha_{1}-\alpha}{2\alpha_{1}}\nabla^{(-\alpha_{1})}
\end{equation*}
holds for all $\alpha\in\mathbf{R}$, the following relation
\begin{eqnarray}
&&	R^{(\alpha)}(X,Y)Z=\frac{\alpha_{1}+\alpha}{2\alpha_{1}}R^{(\alpha_{1})}(X,Y)Z+\frac{\alpha_{1}-\alpha}{2\alpha_{1}}R^{(-\alpha_{1})}(X,Y)Z\nonumber\\
&& \qquad \qquad \qquad \qquad +\frac{\alpha_{1}^{2}-\alpha^{2}}{4\alpha_{1}^{2}}[K(Y,K(Z,X))-K(X,K(Y,Z))]\nonumber
\end{eqnarray}
holds, where $K(X, Y) :=\nabla _{X}Y-\nabla ^{\circ }_{X}Y$  is the difference tensor field of a statistical manifold.

From the relations 
\begin{equation*}
	R^{(\alpha_{1})}(X,Y)Z=k_{1}\{ g(Y,Z)X-g(X,Z)Y\},
\end{equation*}
\begin{equation*}
	R^{(\alpha_{2})}(X,Y)Z=k_{2}\{ g(Y,Z)X-g(X,Z)Y\},
\end{equation*}
the relation
\begin{equation*}
	R^{(\alpha)}(X,Y)Z=\frac{k_{2}\alpha_{1}^{2}-k_{1}\alpha_{2}^{2}+(k_{1}-k_{2})\alpha^{2}}{\alpha_{1}^{2}-\alpha_{2}^{2}}\{ g(Y,Z)X-g(X,Z)Y\}
\end{equation*}
holds, that is, a statistical manifold  $(M, \nabla^{(\alpha)}, g )$ is of constant curvature $\frac{k_{2}\alpha_{1}^{2}-k_{1}\alpha_{2}^{2}+(k_{1}-k_{2})\alpha^{2}}{\alpha_{1}^{2}-\alpha_{2}^{2}}$.
\end{proof}


\noindent\textbf{Example 3.1.}
Let $(M, g )$  be a family of normal distributions:
\begin{equation*}
	M:=\left\{ p(x,\theta ) \left\vert p(x,\theta)=\frac{1}{\sqrt{2\pi (\theta^{2})^{2}}}\exp \left\{ -\frac{1}{2(\theta^{2})^{2}}(x-\theta^{1})^{2}\right\} \right. \right\},\quad g:=2(\theta^{2})^{-2}\sum d\theta^{i}d\theta^{i},
\end{equation*}
\begin{equation*}
	x\in \mathbf{R},\quad \theta^{1}\in\mathbf{R},\quad \theta^{2}>0.
\end{equation*}

We define an $\alpha -$ connection by the following relations:
\begin{equation*}
 \nabla^ {(\alpha)}_{\frac{\partial }{\partial \theta^{1}}}\frac{\partial }{\partial \theta^{1}}=(-1+2\alpha )(\theta^{2})^{-1}\frac{\partial}{\partial \theta^{2}} ,\quad
 \nabla^ {(\alpha)}_{\frac{\partial }{\partial \theta^{2}}}\frac{\partial }{\partial \theta^{2}}=(1+\alpha )(\theta^{2})^{-1}\frac{\partial}{\partial \theta^{2}} ,
\end{equation*}
\begin{equation*}
  \nabla^ {(\alpha)}_{\frac{\partial }{\partial \theta^{1}}}\frac{\partial }{\partial \theta^{2}}=\nabla^ {(\alpha)}_{\frac{\partial }{\partial \theta^{2}}}\frac{\partial }{\partial \theta^{1}}=0.
\end{equation*}
Then the statistical manifold $(M, \nabla^{(0)}, g )$   is of constant curvature $(-\frac{1}{2})$, and the statistical manifold $(M, \nabla^{(1)}, g )$  is of constant curvature 0. Hence for all $\alpha\in\mathbf{R}$, the statistical manifold $(M, \nabla^{(\alpha)}, g )$   is of constant curvature  $\frac{\alpha^{2}-1}{2}$.\\


\noindent\textbf{Example 3.2.}
Let $(M, g )$  be a family of random walk  distributions (\cite{abdel}):
\begin{equation*}
	M:=\left\{ p(x; \theta^{1},\theta^{2} ) \left\vert p(x; \theta^{1},\theta^{2})=\sqrt{\frac{\theta^{2}}{2\pi x}}\exp \left\{ -\frac{\theta^{2}x}{2}+\frac{\theta^{2}}{\theta^{1}}- \frac{\theta^{2}}{2(\theta^{1})^2 x}\right\} ,\quad x,\mu , \lambda >0 \right. \right\},
\end{equation*}
\begin{equation*}
 g:=\frac{\theta^{2}}{(\theta^{1})^{3}}(d\theta^{1})^{2}+\frac{1}{2(\theta^{2})^{2}}(d\theta^{2})^{2}.
\end{equation*}

We define an $\alpha -$ connection by the following relations:
\begin{equation*}
 \nabla^ {(\alpha)}_{\frac{\partial }{\partial \theta^{1}}}\frac{\partial }{\partial \theta^{1}}=\frac{-3(1+\alpha )}{2}(\theta^{1})^{-1}\frac{\partial}{\partial \theta^{1}}+(-1+\alpha )(\theta^{1})^{-3}(\theta^{2})^{2}\frac{\partial}{\partial \theta^{2}} ,
 \end{equation*}
 \begin{equation*}
 \nabla^ {(\alpha)}_{\frac{\partial }{\partial \theta^{1}}}\frac{\partial }{\partial \theta^{2}}= \nabla^ {(\alpha)}_{\frac{\partial }{\partial \theta^{2}}}\frac{\partial }{\partial \theta^{1}}=\frac{(1+\alpha )}{2}(\theta^{2})^{-1}\frac{\partial}{\partial \theta^{1}} ,
\end{equation*}
\begin{equation*}
 \nabla^ {(\alpha)}_{\frac{\partial }{\partial \theta^{2}}}\frac{\partial }{\partial \theta^{2}}= (-1+\alpha )(\theta^{2})^{-1}\frac{\partial}{\partial \theta^{2}}.
\end{equation*}
Then the statistical manifold $(M, \nabla^{(0)}, g )$   is of constant curvature $(-\frac{1}{2})$, and the statistical manifold $(M, \nabla^{(1)}, g )$  is of constant curvature 0. Hence for all $\alpha\in\mathbf{R}$, the statistical manifold $(M, \nabla^{(\alpha)}, g )$   is of constant curvature  $\frac{\alpha^{2}-1}{2}$.

Theorem \ref{theo3_1} implies the following fact. 


\begin{Coro}\label{coro3_1}
  If there exist  $\alpha_{1}, \alpha_{2}\in\mathbf{R} (|\alpha_{1}|\neq|\alpha_{2}|)$   such that the statistical manifold $(M, \nabla^{(\alpha_{1})}, g )$  is of constant curvature $k_{1}$  and the statistical manifold $(M, \nabla^{(\alpha_{2})}, g )$   is of constant curvature  $k_{2}$, and $k_{1}\neq k_{2}$, then for $\alpha\in\mathbf{R}$  satisfying that $\alpha^{2}=(k_{2}\alpha_{1}^2-k_{1}\alpha_{2}^2)/(k_{2}-k_{1})$, the statistical manifold  $(M, \nabla^{(\alpha)}, g )$ is flat.
\end{Coro}


\noindent\textbf{Example 3.3.}
$k_{1}=-1/2$, $k_{2}=0$, $\alpha_{1}=0$  and $\alpha_{2}=1$  hold in example 3.1 and example 3.2. Hence for $\alpha\in\mathbf{R}$  satisfying that $\alpha^{2}=1$, the statistical manifold $(M, \nabla^{(\alpha)}, g )$  is flat.


\begin{Thot}\label{theo3_2}
If the Hessian manifold  $(M, \nabla, g )$ is of constant Hessian curvature, then for  all $\alpha\in\mathbf{R}$, the statistical manifold  $(M, \nabla^{(\alpha)}, g )$ is of constant curvature.
\end{Thot}

\begin{proof}
If the Hessian manifold  $(M, \nabla, g )$ is of constant Hessian curvature, then for all $X,Y,Z\in TM$, 
\begin{equation*}
(\nabla K)(Y, Z;X)=-\frac{c}{2}\{ g(X, Y)Z+g(X, Z)Y\} ,c\in\mathbf{R}
\end{equation*}
holds. On the other hand, the curvature tensor $R^{\circ}$  of Levi-Civita connection  $\nabla^{\circ}$ is written by
\begin{eqnarray}
&&	R^{\circ}(X,Y)Z=R(X,Y)Z-(\nabla K)(Y,Z;X)+(\nabla K)(Z,X;Y)\nonumber\\
&& \qquad \qquad \qquad +K(X,K(Y,Z))-K(Y,K(Z,X)), \nonumber
\end{eqnarray}
where $R$  is the curvature tensor of $\nabla$  and  $K(X, Y) =\nabla _{X}Y-\nabla ^{\circ }_{X}Y$  is difference tensor. 
Then 
\begin{eqnarray}
&& (\nabla K)(Y,Z;X)-(\nabla K)(Z,X;Y)\nonumber\\
&& \qquad =2\{ K(X,K(Y,Z))-K(Y,K(Z,X))\} + \frac {1}{2} \{R(X,Y)Z-R^{*}(X,Y)Z\} \nonumber
\end{eqnarray}
implies
\begin{equation*}
R^{\circ }(X, Y)Z=-\frac{c}{4}\{ g(Y, Z)X-g(X, Z)Y\}, 
\end{equation*}
where $R^{*}$  is curvature tensor of dual connection  $\nabla^{*}$, that is, the statistical manifold $(M, \nabla^{\circ}, g )$ is of constant curvature. On the other hand, the statistical manifold $(M, \nabla, g )$  is flat, that is, constant curvature 0. Therefore we finish the proof of theorem by applying Theorem \ref{theo3_1}.
\end{proof}

Hitherto we found some conditions that for any $\alpha\in\mathbf{R}$, the statistical manifold $(M, \nabla^{(\alpha)}, g )$  is of constant curvature.

                                          
\section{The hypersurfaces of statistical manifolds of constant curvature }
\quad

We consider statistical hypersurfaces of some type of statistical manifolds, which enable for any $\alpha\in\mathbf{R}$  a statistical manifold  $(M, \nabla^{(\alpha)}, g )$ to be of constant curvature.
\begin{Thot}\label{theo4_1}
Let $(M, \nabla, g )$  be a trivial statistical manifold of constant curvature $k$,  $(\tilde M, \tilde\nabla, \tilde g )$ a statistical manifold of constant curvature  $\tilde k$ with a Riemannian manifold of constant curvature  $\stackrel{\circ}{\tilde k}(\neq\tilde k) ~ (\tilde M, \tilde\nabla^{\circ}, \tilde g)$, and  $f:M\rightarrow  \tilde M$  a statistical immersion of codimension one. Then  $f:M\rightarrow  \tilde M$  is equiaffine, that is, $\tau^{*}$  vanishes. 
\end{Thot}
\begin{proof}
If $(\tilde M, \tilde\nabla, \tilde g )$  is a statistical manifold of constant curvature $\tilde k$  with a Riemannian manifold of constant curvature  $\stackrel{\circ}{\tilde k}(\neq\tilde k) ~ (\tilde M, \tilde\nabla^{\circ}, \tilde g)$, the following equation
\begin{eqnarray}
&& (\tilde\nabla_{X} \tilde K)(f_{*}Y,f_{*}Z)-(\tilde\nabla_{Y} \tilde K)(f_{*}X,f_{*}Z) \nonumber \\
&& \qquad \qquad = 2\{ \tilde R(f_{*}X,f_{*}Y)f_{*}Z-{\tilde R}^{\circ}(f_{*}X,f_{*}Y)f_{*}Z\} \\ 
&& \qquad \qquad = 2(\tilde k-\stackrel{\circ}{\tilde k})\{ \tilde g(f_{*}Y,f_{*}Z)f_{*}X-\tilde g(f_{*}X,f_{*}Z)f_{*}Y \} \nonumber
\end{eqnarray}
holds by Eq.(2.2) and Eq.(2.3) in \cite{furu}. By above equation and equation Eq.(3.6) in \cite{furu}, we have
\begin{eqnarray}
&& -2(\tilde k-\stackrel{\circ}{\tilde k}) \{  g(Y,Z)X- g(X,Z)Y \} = (\nabla_{X} K)(Y,Z)-(\nabla_{Y}K)(X,Z)\nonumber\\  
&& \qquad \qquad \qquad \qquad   -b(Y,Z)A^{*}X+b(X,Z)A^{*}Y+h(X,Z)B^{*}Y-h(Y,Z)B^{*}X\nonumber\\
&& 0=(\nabla_{X}b)(Y,Z)-(\nabla_{Y}b)(X,Z)+\tau ^{*}(X)b(Y,Z)-\tau ^{*}(Y)b(X,Z) \\
&& \qquad \qquad \qquad \qquad -\tau ^{*}(Y)h(X,Z)+\tau ^{*}(X)h(Y,Z)\nonumber\\
&& 0=-\tau ^{*}(Y)A^{*}X+-\tau ^{*}(X)A^{*}Y-(\nabla_{X}B^{*})Y+(\nabla_{Y}B^{*})X+\tau ^{*}(X)B^{*}Y-\tau ^{*}(Y)B^{*}X \nonumber\\
&& 0=-h(X,B^{*}Y)+h(Y,B^{*}X)+(\nabla_{X}\tau ^{*})(Y)-(\nabla_{Y}\tau ^{*})(X)+b(Y,A^{*}X)-b(X,A^{*}Y). \nonumber
\end{eqnarray}
By $K=0$, $B^{*}=A^{*}-S$  and Gauss equation (3.3)$_{1}$ in \cite{furu}, from Eq.(4.2)$_{1}$, we have 
\begin{eqnarray*}
&& -2(\tilde k-\stackrel{\circ}{\tilde k}) \{  g(Y,Z)X- g(X,Z)Y \} =-b(Y,Z)A^{*}X+b(X,Z)A^{*}Y \\  
&& \qquad \qquad \qquad \qquad   +h(X,Z)A^{*}Y-h(X,Z)SY-h(Y,Z)A^{*}X+h(Y,Z)SX \\
&& \qquad =-b(Y,Z)A^{*}X+b(X,Z)A^{*}Y-h(X,Z)SY+h(Y,Z)SX+ \tilde R (X,Y)Z-R (X,Y)Z.  
\end{eqnarray*}
By $b=h-II$, $B^{*}=A^{*}-S$  and Riemannian Gauss equation (3.5)$_{1}$ in \cite{furu}, we have
\begin{align*}
& -2(\tilde k-\stackrel{\circ}{\tilde k}) \{  g(Y,Z)X- g(X,Z)Y \} \\  
& \quad =-h(Y,Z)A^{*}X+II(Y,Z)A^{*}X+h(X,Z)A^{*}Y-II(X,Z)A^{*}Y\\
&  \qquad \qquad \qquad \qquad -h(X,Z)SY+h(Y,Z)SX+ \tilde R (X,Y)Z-R (X,Y)Z \\
& \quad =-h(Y,Z)B^{*}X+h(X,Z)B^{*}Y+II(Y,Z)B^{*}X+II(Y,Z)SX \\
&  \qquad \qquad \qquad \qquad -II(X,Z)B^{*}Y-II(X,Z)SY+ \tilde R (X,Y)Z-R (X,Y)Z  \\
& \quad =-b(Y,Z)B^{*}X+b(X,Z)B^{*}Y+ R^{\circ}(X,Y)Z- \tilde R^{\circ}(X,Y)Z+\tilde R (X,Y)Z-R (X,Y)Z. 
\end{align*}

Since  $(M, \nabla , g)$ is Riemannian manifold, clearly $R^{\circ} (X,Y)Z=R (X,Y)Z$. Hence we have
\begin{equation*}
 0=(\tilde k-\stackrel{\circ}{\tilde k}) \{  g(Y,Z)X- g(X,Z)Y \} -b(Y,Z)B^{*}X+b(X,Z)B^{*}Y.
\end{equation*}
And since $b(Y,Z)=g(BY,Z)$, $b(X,Z)=g(BX,Z)$, from above equation we have 

\begin{equation}
 0=(\tilde k-\stackrel{\circ}{\tilde k}) \{  g(Y,Z)X- g(X,Z)Y \} -g(BY,Z)B^{*}X+g(BX,Z)B^{*}Y.
\end{equation}
From Eq.(4.2)$_{3}$, $B^{*}=A^{*}-S$ and Codazzi equation on  $A$ we get
\begin{align}
0 &=-\tau ^{*}(Y)A^{*}X+\tau ^{*}(X)A^{*}Y-(\nabla_{X}A^{*})Y+(\nabla_{X}S)Y+(\nabla_{Y}A^{*})X-(\nabla_{Y}S)X \nonumber\\
& \qquad \qquad \qquad 
+\tau ^{*}(X)B^{*}Y-\tau ^{*}(Y)B^{*}X\nonumber\\
&=(\nabla_{X}S)Y-(\nabla_{Y}S)X+\tau ^{*}(X)B^{*}Y-\tau ^{*}(Y)B^{*}X \nonumber
\end{align}
and by $\nabla=\nabla^{\circ}$  and Codazzi equation on $S$, we also get
\begin{equation}
 0=\tau ^{*}(X)B^{*}Y-\tau ^{*}(Y)B^{*}X.
\end{equation}
From Eq.(4.2)$_{4}$, $B^{*}=A^{*}-S$  and Ricci equation we have 
\begin{equation*}
 b(X,B^{*}Y)-b(Y,B^{*}X)=0,
\end{equation*}
and since $b(X,B^{*}Y)=g(BX,B^{*}Y)$  and $b(Y,B^{*}X)=g(BY,B^{*}X)$, we have
\begin{equation*}
g(BX,B^{*}Y)-g(BY,B^{*}X)=0.
\end{equation*}

Since $g(BX,B^{*}Y)=g(B^{*}Y,BX)=b^{*}(BX,Y)=g(B^{*}BX,Y)$, we have
\begin{equation}
 0=-g([B,B^{*}]X,Y).
\end{equation}
From Eq.(4.5), $B$  and $B^{*}$  are simultaneously diagonalizable. 

In the case that $B^{*}$  is of the form $\lambda^{*}I$, we see easily that $\tau^{*}$  vanishes from Eq.(4.4) if $\lambda^{*}\neq 0$ and $\stackrel{\circ}{\tilde k}=\tilde k$ from Eq.(4.3) otherwise. 
In the case that  $B^{*}$ is not of the form $\lambda^{*}I$, there are $\lambda_{1}^{*}, \lambda_{2}^{*}$ with $\lambda_{1}^{*}\neq \lambda_{2}^{*}$ such  that $B^{*}X_{j}=\lambda_{j}^{*}X_{j}$, where  $g(X_{i},X{_j})=\delta_{ij}, ~ i,j=1,2$. Besides there are $\lambda_{1},\lambda_{2}$  such that $BX_{j}=\lambda_{j}X_{j}$. Eq.(4.3) implies that
\begin{align*}
& (\tilde k-\stackrel{\circ}{\tilde k})\{ g(X_{j},Z)X_{i}-g(X_{i},Z)X_{j}\} +\lambda_{j}\lambda_{i}^{*}g(X_{j},Z)X_{i}-\lambda_{i}\lambda_{j}^{*}g(X_{i},Z)X_{j} \\
& \qquad =(\tilde k-\stackrel{\circ}{\tilde k}+\lambda_{j}\lambda_{i}^{*})g(X_{j},Z)X_{i}-(\tilde k-\stackrel{\circ}{\tilde k}+\lambda_{i}\lambda_{j}^{*})g(X_{i},Z)X_{j} = 0 \\
\end{align*}
and hence $\tilde k-\stackrel{\circ}{\tilde k}+\lambda_{j}\lambda_{i}^{*}=\tilde k-\stackrel{\circ}{\tilde k}+\lambda_{i}\lambda_{j}^{*}=0$, which means that
\begin{equation*}
\lambda_{j}\lambda_{i}^{*}=\lambda_{i}\lambda_{j}^{*}=-(\tilde k-\stackrel{\circ}{\tilde k})\neq 0.
\end{equation*}
By Eq.(4.4) we have $\lambda_{2}^{*}\tau^{*}(X_{1})X_{2}-\lambda_{1}^{*}\tau^{*}(X_{2})X_{1}=0$, which implies that  $\tau^{*}$ vanishes.
\end{proof}


\noindent\textbf{Example 4.1.}
Suppose  $\tilde M$ be $\mathbf{R}^{3}$. We define Riemannian metric and an Affine connection by the following relations:
\begin{equation*}
\tilde g=a\sum d\theta^{i}d\theta^{i},
\end{equation*}
\begin{equation*}
\tilde \nabla_{\frac{\partial }{\partial \theta^{1}}}\frac{\partial }{\partial \theta^{1}}=\tilde b\frac{\partial }{\partial \theta^{1}} ,\tilde \nabla_{\frac{\partial }{\partial \theta^{2}}}\frac{\partial }{\partial \theta^{2}}=\frac{\tilde b}{2}\frac{\partial }{\partial \theta^{1}},\tilde \nabla_{\frac{\partial }{\partial \theta^{3}}}\frac{\partial }{\partial \theta^{3}}=\frac{\tilde b}{2}\frac{\partial }{\partial \theta^{1}},
\end{equation*}
\begin{equation*}
\tilde \nabla_{\frac{\partial }{\partial \theta^{1}}}\frac{\partial }{\partial \theta^{2}}=\tilde \nabla_{\frac{\partial }{\partial \theta^{2}}}\frac{\partial }{\partial \theta^{1}}=\frac{\tilde b}{2}\frac{\partial }{\partial \theta^{2}},\tilde \nabla_{\frac{\partial }{\partial \theta^{1}}}\frac{\partial }{\partial \theta^{3}}=\tilde \nabla_{\frac{\partial }{\partial \theta^{3}}}\frac{\partial }{\partial \theta^{1}}=\frac{\tilde b}{2}\frac{\partial }{\partial \theta^{2}},
\end{equation*}
\begin{equation*}
\tilde \nabla_{\frac{\partial }{\partial \theta^{2}}}\frac{\partial }{\partial \theta^{3}}=\tilde \nabla_{\frac{\partial }{\partial \theta^{3}}}\frac{\partial }{\partial \theta^{2}}=0.
\end{equation*}
Then $(\tilde M,\tilde\nabla,\tilde g)$  is a statistical manifold of constant curvature $-\frac{\tilde b^{2}}{4a}$  with a trivial statistical manifold of constant curvature 0  $(\tilde M,\tilde\nabla^{\circ},\tilde g)$. 
Suppose $M$  be $\mathbf{R}^{2}$, and  $(\nabla,g)$ an induced statistical structure from $(\tilde\nabla,\tilde g)$  by an immersion  $f:(x,y)(\in \mathbf{R}^{2})\mapsto (0,x,y)$. We remark that  $(M,\nabla,g)$ is a trivial statistical manifold of constant curvature 0.

Theorem \ref{theo3_2} and Theorem \ref{theo4_1} imply the following fact.


\begin{Coro}\label{coro4_1}
Let  $(M,\nabla,g)$ be a trivial statistical manifold of constant curvature $k$,  $(\tilde M,\tilde\nabla,\tilde g)$ a Hessian manifold of constant Hessian curvature $\tilde c$, and $f:M\rightarrow  \tilde M$  a statistical immersion of codimension one. Then $f:M\rightarrow  \tilde M$  is equiaffine, that is, $\tau^{*}$  vanishes.
\end{Coro}

We consider a shape operator of statistical immersion of a trivial statistical manifold of constant curvature into a Hessian manifold of constant Hessian curvature.


\begin{lemm}\label{lemm4_1}
Let $(M,\nabla,g)$  be a trivial statistical manifold of constant curvature  $k$, $(\tilde M,\tilde\nabla,\tilde g)$ a Hessian manifold of constant Hessian curvature $\tilde c$, and  $f:M\rightarrow  \tilde M$ a statistical immersion of codimension one. Then the following holds: 
\begin{equation*}
A^{*}=k\nu\tilde c^{-1}I,B^{*}=-\frac{1}{2}\nu I,h=\tilde c\nu^{-1}g, A=\tilde c\nu^{-1}I, B=[2\tilde c^{2}-(2k+\tilde c)\nu^{2}](2\nu \tilde c)^{-1}I.
\end{equation*}
\end{lemm} 


\begin{proof} 
Combining Eq.(2.3) and Eq.(3.6) in \cite{furu} with Eq.(2.1), we have
\begin{eqnarray}
&& \frac{\tilde c}{2} \{  g(Y,Z)X- g(X,Z)Y \} =2( k-\stackrel{\circ} k) \{  g(Y,Z)X- g(X,Z)Y \}  \nonumber\\  
&& \qquad \qquad  \qquad   -b(Y,Z)A^{*}X+b(X,Z)A^{*}Y+h(X,Z)B^{*}Y-h(Y,Z)B^{*}X\nonumber\\
&& 0=h(X,K(Y,Z))-h(Y,K(X,Z))+(\nabla_{X}b)(Y,Z)-(\nabla_{Y}b)(X,Z)\nonumber\\ 
&& \qquad \qquad  \qquad   +\tau^{*}(X)b(Y,Z)-\tau^{*}(Y)b(X,Z)-\tau^{*}(Y)h(X,Z)+\tau^{*}(X)h(Y,Z) \\
&& 0=K(Y,A^{*}X)-K(X,A^{*}Y)-\tau^{*}(Y)A^{*}X+\tau^{*}(X)A^{*}Y\nonumber\\  
&& \qquad \qquad  \qquad   -(\nabla_{X}B^{*})Y+(\nabla_{Y}B^{*})X+\tau^{*}(X)B^{*}Y-\tau^{*}(Y)B^{*}X\nonumber\\
&& 0=-h(X,B^{*}Y)+h(Y,B^{*}X)+(\nabla_{X}\tau^{*})(Y)-(\nabla_{Y}\tau^{*})(X)+b(Y,A^{*}X)-b(X,A^{*}Y)\nonumber 
\end{eqnarray}
Taking the trace of (4.6)$_{1}$ with respect to $X$, we have
\begin{equation*}
-\tilde c g(Y,Z)=-tr A^{*}b(Y,Z)+h(B^{*}Z,Y)+h(B^{*}Y,Z)
\end{equation*}
and taking the trace of (4.6)$_{1}$  with respect to $Y$, we have 
\begin{equation*}
-\frac{\tilde c}{2}(n+1)g(X,Z)=-b(A^{*}X,Z)+h(X,B^{*}Z)+tr B^{*} h(X,Z) .
\end{equation*}
Using the above equation and Eq.(4.6)$_{4}$, we have 
\begin{eqnarray}
 && -\frac{\tilde c}{2}(n+2)g(X,Y)=-b(A^{*}X,Y)+h(X,B^{*}Y)+tr B^{*} h(Y,Y) .\nonumber\\
 && \qquad  \qquad  \qquad  \qquad \qquad -h(X,Y)\nu -h(X,B^{*}Y)+(\nabla_{X}\tau^{*})Y+b(Y,A^{*}X)\nonumber\\
&&  \qquad  \qquad  \qquad \qquad =tr B^{*}h(X,Y)-h(X,Y)\nu +(\nabla_{x}\tau^{*})Y\nonumber
\end{eqnarray}
and  since from Corollary \ref{lemm4_1} $\tau^{*}=0$ holds, we have 
\begin{equation*}
(\nu -tr B^{*})h(X,Y)=\frac{\tilde c}{2}(n+2)g(X,Y).
\end{equation*}
Hence we have
\begin{equation}
h=\frac{\tilde c}{2}(n+2)(\nu -tr B^{*})^{-1}g.
\end{equation}

If $\tilde c \neq 0$  holds, $h$  is non-degenerated.

Since $\tilde \nabla$  is flat in Gaussian equation in \cite{furu}, we obtain
\begin{equation*}
k\{ g(Y,Z)X-g(X,Z)Y\} =h(Y,Z)A^{*}X-h(X,Z)A^{*}Y
\end{equation*}
and  taking the trace of above equation with respect to $X$, we have
\begin{equation*}
k(n-1)g(Y,Z)=tr A^{*} h(Y,Z)-h(A^{*}Y,Z)=h((tr A^{*}I-A^{*})Y,Z) .
\end{equation*}

Since the above equation and Eq.(4.7) imply that 
\begin{equation*}
k(n-1)I=\frac{\tilde c}{2}(n+2)(\nu-tr B^{*})^{-1}(tr A^{*}I-A^{*}),
\end{equation*}
there is  $a\in\mathbf{R}$ such that $A^{*}=aI$  and $tr A^{*}=an$.
Therefore the above equation implies that 
\begin{equation*}
k(n-1)I=\frac{\tilde c}{2}(n+2)(\nu-tr B^{*})^{-1}(na-a)I
\end{equation*}
and thus since
\begin{equation*}
2k(\nu -tr B^{*})=\tilde c(n+2)a,
\end{equation*}
we have
\begin{equation}
A^{*}=2k(\nu -tr B^{*})[\tilde c(n+2)]^{-1}I.
\end{equation}

If $k\neq 0$  holds,then since $A^{*}$  is non-degenerated, by Eq.(4.8) we have
\begin{equation*}
B^{*}=-\frac{\nu}{2}I, ~ tr B^{*}=-\frac{n\nu}{2}
\end{equation*}
and
\begin{equation*}
A^{*}=\frac{2k(\nu +\frac{n\nu}{2})}{\tilde c(n+2)}I=\frac{k\nu}{\tilde c}I, h=\frac{\tilde c}{2}(n+2)(\nu +\frac{n\nu}{2})^{-1}g=\frac{\tilde c}{\nu}g .
\end{equation*}
Since $h(X,Y)=g(AX,Y)$, we have $A=\frac{\tilde c}{\nu}I$  and
\begin{equation*}
B=B^{*}+(A-A^{*})=-\frac{\nu}{2}I +(\frac{\tilde c}{\nu}-\frac{k\nu}{\tilde c})I=\frac{-\nu^{2}\tilde c+2\tilde c^{2}-2k\nu^{2}}{2\nu \tilde c}I=\frac{2\tilde c^{2}-(2k+\tilde c)\nu^{2}}{2\nu\tilde c}I.
\end{equation*}
\end{proof}


\begin{Thot}\label{theo4_2}
Let $(M, \nabla, g )$ be a trivial statistical manifold of constant curvature $k$, $(\tilde M,\tilde\nabla,\tilde g)$  a Hessian manifold of constant Hessian curvature $\tilde c$. If there is a statistical immersion of codimension one $f:M\rightarrow  \tilde M$,  $2k+\tilde c$ is of non-negative. Moreover, when $\tilde c$  is positive, the Riemannian shape operator of $f:M\rightarrow  \tilde M$   is given by $S=\pm \frac{1}{2}\sqrt{2k+\tilde c}I$.
\end{Thot}

\begin{proof} 
By Lemma \ref{lemm4_1} and Eq.(4.2), we have
\begin{eqnarray*}
&& \frac{\tilde c}{4} \{ g(Y,Z)X-g(X,Z)Y\} +\frac{2\tilde c^{2}-(2k+\tilde c)\nu^{2}}{2\nu\tilde c}(-\frac{\nu}{2})\{ g(Y,Z)X-g(X,Z)Y\} \\
&& \qquad = \left[\frac{\tilde c}{4}-\frac{2\tilde c^{2}-(2k+\tilde c)\nu^{2}}{4\tilde c} \right]\{ g(Y,Z)X-g(X,Z)Y\} = 0 
\end{eqnarray*}
and thus conclude that 
\begin{equation*}
\frac{\tilde c}{4}-\frac{2\tilde c^{2}-(2k+\tilde c)\nu^{2}}{4\tilde c}=0 .
\end{equation*}
Since $\tilde c^{2}=(2k+\tilde c)\nu^{2}$, we have $2k+\tilde c\geq 0$  and
\begin{equation*}
\nu=\pm \frac{|\tilde c|}{\sqrt{2k+\tilde c}}.
\end{equation*}
Thus the Riemannian shape operator $S$  is given by 
\begin{equation*}
S=A^{*}-B^{*}=(\frac{k\nu}{\tilde c}+\frac{\nu}{2})I=\frac{2k+\tilde c}{2\tilde c}(\pm \frac{|\tilde c|}{\sqrt{2k+\tilde c}})I=\pm \frac{|\tilde c|}{2\tilde c}\sqrt{2k+\tilde c}I.
\end{equation*}
When $\tilde c$  is positive, we have $S=\pm \frac{1}{2}\sqrt{2k+\tilde c}I$.
\end{proof}


\noindent\textbf{Example 4.2.}
Let $(H, \tilde\nabla, \tilde g)$  be the  $(n+1)-$dimensional upper half Hessian space of constant Hessian curvature 4 as in Example 2.1. For a constant $y_{0}>0$, write the following immersion by $f$: 
\begin{equation*}
(y^{1},\cdot\cdot\cdot ,y^{n})^{T}(\in\mathbf{R}^{n})\mapsto (y^{1},\cdot\cdot\cdot ,y^{n},y_{0})^{T}\in H .
\end{equation*}
Let $(\nabla, g)$  be the statistical structure on $\mathbf R^{n}$  induced by $f$  from $(\tilde\nabla,\tilde g)$. Then $(\mathbf R^{n}, \nabla, g)$  is a trivial statistical manifold of constant curvature 0 and  $f$ is a statistical immersion of a trivial statistical manifold of constant curvature into Hessian manifold of constant Hessian curvature. \\

\textbf{Acknowledgment}
The authors would like to thank the anonymous referees for their helpful comments and suggestions.

\end{document}